\documentclass[12pt, reqno]{amsart}
\usepackage{amsmath, amsthm, amscd, amsfonts, amssymb, graphicx, xcolor}
\usepackage[bookmarksnumbered, colorlinks, plainpages]{hyperref}
\usepackage{amsmath}
\usepackage{cases}
\usepackage{amssymb}
\usepackage{mathrsfs}
\usepackage{euscript}
\usepackage{dsfont}
\usepackage{graphicx}
\usepackage{float}
\usepackage{graphicx}
\usepackage{amssymb}

\allowdisplaybreaks[4]
\newtheorem{theorem}{Theorem}[section]

\newtheorem{lemma}{Lemma}[section]

\newtheorem{corollary}[theorem]{Corollary}
\theoremstyle{definition}

\theoremstyle{remark}

\numberwithin{equation}{section}

\begin{document}
	\setcounter{page}{1}
	

	
	\title[]{The lower bound of shifted primes with large prime factors}
	
	\author[]{Zhiyuan Yang}
	
	\address{School of Mathematics,  Shandong University, Jinan 250100, Shandong, China}
	\email{\tt zhiyuan.yang@mail.sdu.edu.cn}

	
	

	
	
	\begin{abstract}
		In this paper, we consider the asymptotic density of $T_c(x):=\#\{p\leq x:P^+(p-1)\geq p^c\}$, where $P^+(n)$ denote the largest prime factor of $n$. We show that for $x\rightarrow\infty$, for any $0<c\leq 1/2$, one has
		\begin{align*}
			\mathop{\lim\inf}_{x\rightarrow\infty}\frac{T_c(x)}{\pi(x)}&\geq\max\left(1-\frac{16}{5}\rho\left(\frac{1}{c}\right),1-c+g(c)\right),
		\end{align*}
		where $g(c)>0$ for $11/32<c\leq1/2$, and $g(c)=0$ for $0<c\leq11/32$. In particular, we have $g(1/2)>0.0436$. This improves a previous result by Liu-Wu-Xi (2020), who showed
		\begin{align*}
			\mathop{\lim\inf}_{x\rightarrow\infty}\frac{T_c(x)}{\pi(x)}&\geq\max\left(1-4\rho\left(\frac{1}{c}\right),1-c\right),
		\end{align*} 
		for $0<c\leq 1/2$.
		\newline
		\noindent 
	\end{abstract} \maketitle
	
	
	\section{Introduction}
	Let $P^+(n)$ denote the largest prime factors of $n$ with the convention $P^+(1)=1$.
	Define
	\begin{align*}
		T_c(x):=\#\{p\leq x:P^+(p-1)\geq p^c\},
	\end{align*}
	and
	\begin{align*}
		d(c):=\mathop{\lim\inf}_{x\rightarrow \infty}\frac{T_c(x)}{\pi(x)}.
	\end{align*}
	where $0<c<1$.
	
	Consider $1/2\leq c<1$. In [9], Goldfeld proved that $d(1/2)\geq 1/2$. In 1985, Fouvry [8] showed that $d(0.6687)>0$. The best published result is $d(0.677)>0$ due to Baker and Harman [1].
	
	Now consider $0<c<1/2$.
	In 2015, Luca, Menares and Pizarro-Madariaga [13]
	gave $d(c)\geq 1-c$ for $1/4\leq c\leq1/2$ with a detailed error term.
	In 2017, Chen and Chen [2] extended the range of $c$ to $(0,1/2]$.
	In 2018, Feng and Wu [7] improved the lower bound for $0<c< 0.3517\cdots$ by showing
	\begin{align*}
		d(c)\geq 1-4\int_{1/c-1}^{1/c}\frac{\rho(t)}{t}dt,
	\end{align*}
	where $\rho(u)$ is the Dickman function. 
	Later, the above lower bound was further improved by Liu, Wu and Xi [12] to
	\begin{align*}
	    d(c)\geq 1-4\rho(1/c)
	\end{align*}
	provided $0< c<0.3734\cdots$. 
	
	For the upper bound, we define
	\begin{align*}
		u(c):=\mathop{\lim\sup}_{x\rightarrow\infty}\frac{T_c(x)}{\pi(x)}.
	\end{align*}
	We also show the previous results as following. In 1935, Erd\H{o}s [6] showed $u(c)\rightarrow 0$ as $c\rightarrow 1$. In 2023, Ding [3] proved $u(c)\leq (1-2\delta)/(2c)$, for some $\delta>0$, for $3/4<c<1$. In 2025, Ding [4] proved $u(c)\leq 8(c^{-1}-1)$ for $8/9<c<1$. Latter, Ding and Wang [5] showed $u(c)\leq \frac{7}{2}\log\frac{1}{c}$ for $e^{-2/7}<c<1$. Very recently, the author [16, 17] showed
	\begin{align*}
		u(c)\leq \min\left(\frac{1-\delta}{2c},\frac{16}{5}\log\frac{1}{c}\right), 
	\end{align*}
	for some $\delta>0$, for $1/2<c<1$.
	
	In this paper, our results are as follows.
	\begin{theorem}
		For any $0<c\leq1/2$, we have
		\begin{align*}
			\mathop{\lim\inf}_{x\rightarrow\infty}\frac{T_c(x)}{\pi(x)}&\geq\max\left(1-\frac{16}{5}\rho\left(\frac{1}{c}\right),1-c+g(c)\right),
		\end{align*}
		where $g(c)>0$ for $11/32<c\leq1/2$, and $g(c)=0$ for $0<c\leq11/32$. In particular, we have $g(1/2)>0.0436$.
		
	\end{theorem}
	
	Define
	\begin{align*}
		\pi(x,y;q,a):=\sum_{\substack{p\leq x\\p\equiv a\mkern-15mu\pmod{q}\\P^+((p-a)/q)\leq y}}1.
	\end{align*}
	\begin{corollary}
		For any fixed $a$, we have
		\begin{align*}
			\mathop{\lim\sup}_{x\rightarrow \infty}\frac{\sum_{\substack{q\leq Q\\(q,a)=1}}\Lambda(q)\pi(x,y;q,a)}{\pi(x)\sum_{\substack{q\leq Q\\(q,a)=1}}\frac{\Lambda(q)}{\varphi(q)}\rho\left(\frac{\log(x/q)}{\log y}\right)}\leq \frac{16}{5},
		\end{align*}
		for $x\geq y\geq x^\varepsilon$ and $Q<x^{1-\varepsilon}$.
	\end{corollary}
	This also holds for many functions other than $\Lambda$. Corollary 1.2 improves a previous upper bound $4$ due to Liu-Wu-Xi. (see [12, Theorem 1.2 (i)])

	\section{Lemmas}

	\begin{lemma}
		For $\varepsilon>0$, we have
		\begin{align*}
			\Psi(x,y):=\sum_{\substack{n\leq x\\P^+(n)\leq y}}1=x\rho(u)\left(1+O_\varepsilon\left(\frac{\log(u+1)}{\log y}\right)\right)
		\end{align*}
		uniformly for
		\begin{align*}
			x\geq x_0(\varepsilon),\qquad \exp\{(\log\log x)^{5/3+\varepsilon}\}\leq y\leq x,
		\end{align*}
		where $u=\log x/\log y$ and $\rho(u)$ is the Dickman function.
	\end{lemma}
	\begin{proof}
		The is is [10, Theorem 1]. 
	\end{proof}
	
	\begin{lemma}
		Let $D\geq 2$ and $L>1$. Let $\mathscr{P}$ denote a set of primes. Let $z\geq 2$ and write $P(z):=\prod_{p\leq z,p\in\mathscr{P}}p$. There exist two sequences $\{\lambda_d^\pm\}_{d=1}^{\infty}$ of real numbers, vanishing for $d>D$ or $\mu(d)=0$, satisfying $\lambda_1^\pm=1$, $|\lambda_d^{\pm}|=O(1)$, $\lambda^-*1\leq \mu*1\leq \lambda^+*1$, and such that
		\begin{align*}
			&\sum_{d|P(z)}\lambda_d^+\frac{\omega(d)}{d}\leq \prod_{\substack{p\leq z\\p\in \mathscr{P}}}\left(1-\frac{\omega(p)}{p}\right)\left(F(s)+o(1)\right),\\&\sum_{d|P(z)}\lambda_d^-\frac{\omega(d)}{d}\geq \prod_{\substack{p\leq z\\p\in \mathscr{P}}}\left(1-\frac{\omega(p)}{p}\right)\left(f(s)+o(1)\right)
		\end{align*}
		uniformly for all multiplicative function $\omega$ satisfying
		\begin{align*}
			&(i)\ 0<\omega(p)<p\ (p\in\mathscr{P}),\\&(ii)\ \prod_{\substack{u<p\leq v\\p\in\mathscr{P}}}\left(1-\frac{\omega(p)}{p}\right)^{-1}\leq \frac{\log v}{\log u}\left(1+\frac{L}{\log u}\right)\ (2\leq u\leq v\leq z),
		\end{align*}
		where $s=\log D/\log z=O(1)$ and $f(s)$, $F(s)$ are determined by the differential-difference equations
		\begin{equation}
			\begin{aligned}
				&\left\{ \begin{aligned}
					&sF(s)=2e^{\gamma}, \quad&& 1\leq s\leq 2;\\&sf(s)=0,\quad&& 0\leq s\leq 2;\\&(sF(s))'=f(s-1),(sf(s))'=F(s-1),\quad&&s\geq 2.
				\end{aligned}\right.\nonumber
			\end{aligned}
		\end{equation}
		In particular, the functions $\lambda^{\pm}$ can be rewrited as
		\begin{align*}
			\lambda^{\pm}=\sum_{h<\exp(8\varepsilon^{-3})}\lambda^{\pm}(\cdot,h),
		\end{align*}
		where $\lambda^{\pm}(\cdot,h)$ is a well-factorable function of level $D$.
	\end{lemma}
	\begin{proof}
		This lemma is the Rosser-Iwaniec sieve [11].
	\end{proof}

	\begin{lemma}
		Define
		\begin{align*}
			H(n):=\prod_{p>2,p|n}\left(\frac{p-1}{p-2}\right),
		\end{align*}
		Then for $y>x^{\varepsilon}$, we have
		\begin{align*}
			\sum_{\substack{n\leq x,P^+(n)\leq y\\(n,a)=1}}H(n)=\frac{c\Psi(x,y)}{2^{\varepsilon(a)}H(a)}(1+o(1)),
		\end{align*}
		where 
		\begin{align*}
			c=\prod_{p>2}\left(1+\frac{1}{p(p-2)}\right).
		\end{align*}
	\end{lemma}
	\begin{proof}
		We can see it in [12, (3.5)].
	\end{proof}
	\subsection*{Convention.} We use $\varepsilon$ to denote a sufficiently small positive number, and the value of $\varepsilon$ may change from statement to statement. We use $\mu(n)$ and $\varphi(n)$ to denote the M$\mathrm{\ddot{o}}$bius function and Euler's function, respectively. We use $\rho(u)$ to denote the Dickman function, respectively. By $(m_1,m_2)$ we denote the greatest common divisor of $m_1$ and $m_2$. We use the standard asymptotic notation $f\ll g,f\gg g,f=O(g),f=o_{x\rightarrow\infty}(g)$ from analytic number theory, and indicate that the implicit constants depend on some parameter $\varepsilon$ through subscripts. By $m\sim M$ we denote $M< m\leq2M$. Define
	\begin{align*}
		P(x):=\prod_{p\leq x}p,\quad P(y,z):=\prod_{y<p\leq z}p.
	\end{align*}

	\section{Proof of Theorem 1.1}

	We apply the following lemma to estimate $T_c(x)$.\begin{lemma}
		For $0<c<1$ and sufficiently large $x$, we have
		\begin{align*}
			\sum_{\substack{p\leq x\\P^+(p-1)\geq p^c}}1=\sum_{\substack{p\leq x\\P^+(p-1)\geq x^{c}}}1+O\left(\frac{x\log\log x}{(\log x)^2}\right).
		\end{align*}
	\end{lemma}\begin{proof}
		This is [15, Theorem 2].
	\end{proof}
	We start from the following expression
	\begin{align*}
		\sum_{\substack{p\leq x\\P^+(p-1)\geq x^c}}1=\sum_{p\leq x}1-\sum_{\substack{p\leq x\\P^+(p-1)<x^c}}1=(1+o(1))\pi(x)-\mathcal{T}_c(x).\tag{3.1}
	\end{align*}
	We only need to estimate $\mathcal{T}_c(x)$.
	
	\subsection{Case 1: $0<c\leq1/2$}
	Fix $0<\delta<1/100$. Let $y=x^{c\varepsilon/(3J)}$, where $J=J(\delta)$ is a parameter satisfying
	\begin{align*}
		\frac{(\log(1/\delta))^{J+1}}{(J+1)!}\leq \delta.
	\end{align*}
	Since
	\begin{align*}
		\sum_{\ell}=\sum_{\substack{\ell\\(\ell,P(y^\delta,y))>1}}+\sum_{\substack{\ell\\(\ell,P(y^\delta,y))=1}},
	\end{align*}
	we write
	\begin{align*}
		\mathcal{T}_c(x)=\sum_{\substack{\ell\leq x\\P^+(\ell)<x^c\\\ell+1\ \text{is prime}}}1=\sum_{\substack{\ell\leq x\\P^+(\ell)<x^c\\(\ell,P(y^\delta,y))>1\\\ell+1\ \text{is prime}}}1+\sum_{\substack{\ell\leq x\\P^+(\ell)<x^c\\(\ell,P(y^\delta,y))=1\\\ell+1\ \text{is prime}}}1=\mathscr{S}_1+\mathscr{S}_2.\tag{3.2}
	\end{align*}
	\subsubsection{Estimate of $\mathscr{S}_1$}
	
	Note that for any real sequence $(a_\ell)$ with $a_{\ell}\geq 0$, we have
	\begin{align*}
		\sum_{\substack{(\ell,P(y^\delta,y))>1}}a_{\ell}&=\sum_{y^\delta<p_1\leq y}\sum_{\substack{(\ell,P(y^\delta,y))=1}}a_{p_1\ell}+\frac{1}{2!}\sum_{y^\delta<p_1,p_2\leq y}\sum_{\substack{(\ell,P(y^\delta,y))=1}}a_{p_1p_2\ell}\\&\quad+\frac{1}{3!}\sum_{y^\delta<p_1,p_2,p_3\leq y}\sum_{\substack{(\ell,P(y^\delta,y))=1}}a_{p_1p_2p_3\ell}+\cdots+O\Bigg(\sum_{y^{\delta}<p\leq y}\sum_{\ell}a_{p^2\ell}\Bigg)\\&\leq \sum_{1\leq j\leq J}\frac{1}{j!}\sum_{y^\delta<p_1,\cdots,p_j\leq y}\sum_{\substack{(\ell,P(y^\delta,y))=1}}a_{p_1\cdots p_j\ell}+\frac{1}{(J+1)!}\sum_{y^\delta<p_1,\cdots,p_{J+1}\leq y}\sum_{\substack{p_1\cdots p_{J+1}|\ell}}a_{\ell}\\&\quad+O\Bigg(\sum_{y^{\delta}<p\leq y}\sum_{\ell}a_{p^2\ell}\Bigg).
	\end{align*}
	Let $D=z^2=x^{5/8-\varepsilon}$ and $P(z)=\prod_{p\leq z}p$.
	Define
	\begin{align*}
		\mathcal{L}:=\{\ell\leq x:P^+(\ell)<x^c,(\ell,P(y^\delta,y))=(\ell,P^2(y^\delta,y))>1,2|\ell\}.
	\end{align*}
	By Lemma 2.2, there exists a sequence
	$\{\lambda_{d}^+\}_{d=1}^{\infty}$ of real numbers, vanishing for $d>D$ or $\mu(d)=0$, satisfying $\lambda_1^+=1$, $|\lambda_{d}^+|=O(1)$, $0\leq \mu*1\leq \lambda^+*1$. In particular, this sequence $\lambda^+$ can be expressed as a sum of well-factorable sequences.
	We have
	\begin{align*}
		\mathscr{S}_1&\leq  \sum_{\substack{ \ell\in\mathcal{L}\\(\ell+1,P(z))=1}}1+O(z)+O\Bigg(\sum_{y^{\delta}<p_1\leq y}\sum_{\substack{p\leq x\\p\equiv 1\mkern-15mu\pmod{p_1^2}}}1\Bigg)\\&\leq \sum_{\substack{ \ell\in\mathcal{L} }}\sum_{d|(\ell +1,P(z))}\lambda_d^++o(\pi(x))\\&\leq \sum_{1\leq j\leq J}\frac{1}{j!}\sum_{y^\delta<p_1,\cdots,p_j\leq y}\sum_{\substack{ \ell\leq x\\p_1\cdots p_j|\ell\\(\ell(p_1\cdots p_j)^{-1},P(y^\delta,y))=1\\P^+(\ell)<x^c\\2|\ell }}\sum_{d|(\ell +1,P(z))}\lambda_d^+\\&\quad +\frac{1}{(J+1)!}\sum_{y^\delta<p_1,\cdots,p_{J+1}\leq y}\sum_{\substack{ \ell\leq x\\p_1\cdots p_{J+1}|\ell\\P^+(\ell)<x^c\\2|\ell }}\sum_{d|(\ell +1,P(z))}\lambda_d^++o(\pi(x)).\tag{3.3}
	\end{align*}
	For any $1\leq j\leq J$, changing the order of summation, we write
	\begin{align*}
		&\sum_{y^\delta<p_1,\cdots,p_j\leq y}\sum_{\substack{ \ell\leq x\\p_1\cdots p_j|\ell\\(\ell(p_1\cdots p_j)^{-1},P(y^\delta,y))=1\\P^+(\ell)<x^c\\2|\ell }}\sum_{d|(\ell +1,P(z))}\lambda_d^+\\&=\sum_{d|P'(z)}\lambda_d^+\sum_{y^\delta<p_1,\cdots,p_j\leq y}\sum_{\substack{ \ell\leq x\\p_1\cdots p_j|\ell\\(\ell(p_1\cdots p_j)^{-1},P(y^\delta,y))=1\\P^+(\ell)<x^c\\2|\ell,(\ell,d)=1 }}\frac{1}{\varphi(d)}\\&\quad+\sum_{d|P'(z)}\lambda_d^+\sum_{y^\delta<p_1,\cdots,p_j\leq y}\sum_{\substack{ \ell\leq x\\p_1\cdots p_j|\ell\\(\ell(p_1\cdots p_j)^{-1},P(y^\delta,y))=1\\P^+(\ell)<x^c\\2|\ell }}\Bigg(1_{\ell \equiv -1\mkern-15mu\pmod{d}}-\frac{1_{(\ell,d)=1}}{\varphi(d)}\Bigg)\\&=M_j+R_j,
	\end{align*}
	where $P'(z)=\prod_{2<p\leq z}p$.
	Write $\ell=p_1\cdots p_j\ell_1$ with $(\ell_1,P(y^\delta,y))=1$ and $P^+(\ell_1)<x^c$.
	In order to prove $R_j\ll x(\log x)^A$, putting $p_1$ and $e=p_2\cdots p_j \ell_1$ in dyadic range, we only need to prove that for any $A>1$
	\begin{align*}
		\sum_{d}\lambda_d\sum_{\substack{p_1\in\mathscr{P}_1}}\sum_{e\in\mathscr{E}}b_e\Bigg(1_{p_1e\equiv a\mkern-15mu\pmod{d}}-\frac{1_{(p_1e,d)=1}}{\varphi(d)}\Bigg)\ll\frac{x}{(\log x)^A},\tag{3.4}
	\end{align*}
	where $\lambda_d$ is a well factorable function of level $D$, $|b_e|\leq 1$, $\mathscr{P}_1=[(1-\Delta)P_1,P_1]$, $\mathscr{E}=[(1-\Delta)E,E]$ with $P_1E\asymp x$, $y^\delta<P_1\leq y$, $\Delta=(\log x)^{-A_1}$. ($A_1$ is a suitable constant.)
	
	We need the following lemma due to Pascadi. (see [14, Proposition 4.4])
	\begin{lemma}
		Let $a\in\mathbb{Z}\setminus\{0\}$, $A,\varepsilon>0$, and $M,N,x,Q_1,Q_2,Q_3\gg1$ satisfy $MN\asymp x$, $N>x^\varepsilon$. 
		Let $(\alpha_n)$, $(\beta_m)$ be $1$-bounded complex sequences, such that $(\alpha_n)$ is supported on $P^-(n)\geq z_0:=x^{1/(\log\log x)^3}$ and satisfies the Siegel-Walfisz condition, which means
		\begin{align*}
			\Bigg|\sum_{\substack{n\sim N\\n\equiv b\mkern-15mu\pmod{q}\\(n,d)=1}}\alpha_n-\frac{1}{\varphi(q)}\sum_{\substack{n\sim N\\(n,dq)=1}}\alpha_n\Bigg|\ll_A\frac{N\tau(d)^{O(1)}}{(\log N)^A},
		\end{align*}
		for any $d\geq 1, q\geq 1$ and $(b,q)=1$, $A>1$.
		If 
		\begin{align*}
			Q_1&\leq Nx^{-\varepsilon} ,\\N^2Q_2Q_3^2&\leq x^{1-15\varepsilon},\\N^{2}Q_2^5Q_3^2&\leq x^{2-40\varepsilon},
		\end{align*} then for any $1$-bounded complex sequences $(\gamma_{q_1})$, $(\lambda_{q_2})$, $(\nu_{q_3})$ supported on $(q_i,a)=1$, one has\begin{align*}
			\sum_{q_1\sim Q_1}\gamma_{q_1}\sum_{q_2\sim Q_2}\lambda_{q_2}\sum_{q_3\sim Q_3}\nu_{q_3}\sum_{n\sim N}\alpha_n\sum_{m\sim M}\beta_m\Bigg(1_{mn\equiv a\mkern-15mu\pmod{q}}-\frac{1_{(mn,q)=1}}{\varphi(q)}\Bigg)\ll_{\varepsilon,A,a}\frac{x}{(\log x)^A}.
		\end{align*}
	\end{lemma}
	
	Applying Lemma 3.2 with $Q_1=1$, $\varepsilon\leftarrow c\delta\varepsilon/(3J)$, $\alpha_n=1_{n\ \text{is prime}}1_{ n\in \mathscr{P}_1}$,
	we obtain that (3.4) holds by writing $\lambda=\lambda_2*\lambda_3$, where $\lambda_2(q_2)$ is supported on $q_2\leq x^{1/4-\varepsilon}$, $\lambda_3(q_3)$ is supported on $q_3\leq x^{3/8}$. We have controlled the remainder term.
	
	Now we begin to estimate the main term. We have
	\begin{align*}
		M_j&=\sum_{d|P'(z)}\lambda_d^+\sum_{y^\delta<p_1,\cdots,p_j\leq y}\sum_{\substack{ \ell\leq x\\p_1\cdots p_j|\ell\\(\ell(p_1\cdots p_j)^{-1},P(y^\delta,y))=1\\P^+(\ell)<x^c\\2|\ell,(\ell,d)=1 }}\frac{1}{\varphi(d)}\\&=\sum_{y^\delta<p_1,\cdots,p_j\leq y}\sum_{\substack{ \ell\leq x^{1-c}(p_1\cdots p_j)^{-1}\\(\ell,P(y^\delta,y))=1\\P^+(\ell)<x^c\\2|\ell }}\sum_{d|P''(z)}\frac{\lambda_d^+}{\varphi(d)},
	\end{align*}
	where $P''(z)=\prod_{2<p\leq z,(p,p_1\cdots p_j\ell)=1}p$. By Lemma 2.2, we have
	\begin{align*}
		\sum_{d|P''(z)}\frac{\lambda_d^+}{\varphi(d)}&\leq (F(2)+o(1))\prod_{\substack{2<p\leq z\\(p,p_1\cdots p_j\ell)=1}}\left(1-\frac{1}{p-1}\right)\\&=(e^\gamma+o(1))\prod_{\substack{2<p\leq z}}\left(\frac{p-2}{p-1}\right)H(p_1\cdots p_j\ell),
	\end{align*}
	where
	\begin{align*}
		H(n)=\prod_{p>2,p|n}\left(\frac{p-1}{p-2}\right).
	\end{align*}
	Mertens' theorem implies
	\begin{align*}
		\prod_{2<p\leq z}\left(\frac{p-2}{p-1}\right)=\frac{2e^{-\gamma}A_0+o(1)}{\log z},
	\end{align*}
	where
	\begin{align*}
		A_0=\prod_{p>2}\left(1-\frac{1}{(p-1)^2}\right).
	\end{align*}
	Then we have
	\begin{align*}
		M_j\leq \frac{4A_0+o(1)}{\log D}\sum_{y^\delta<p_1,\cdots,p_j\leq y}\sum_{\substack{ \ell\leq x(2p_1\cdots p_j)^{-1}\\(\ell,P(y^\delta,y))=1\\P^+(\ell)<x^c}}H(p_1\cdots p_j\ell),\tag{3.5}
	\end{align*}
	Summing over all $1\leq j\leq J$, by partial summation and Lemmas 2.1, 2.3, we have
	\begin{align*}
		\sum_{1\leq j\leq J}\frac{M_j}{j!}&\leq \frac{4A_0+o(1)}{\log D}\sum_{1\leq j\leq J}\frac{1}{j!}\sum_{y^\delta<p_1,\cdots,p_j\leq y}\sum_{\substack{ \ell\leq x/2\\p_1\cdots p_j|\ell\\(\ell(p_1\cdots p_j)^{-1},P(y^\delta,y))=1\\P^+(\ell)<x^c}}H(\ell)\\&\leq \frac{4A_0+o(1)}{\log D}\sum_{\substack{ \ell\leq x/2\\P^+(\ell)<x^c}}H(\ell)\\&=\left(\frac{2}{5/8-\varepsilon}\rho\left(\frac{1}{c}\right)+o(1)\right)\pi(x).\tag{3.6}
	\end{align*}
	Similar to (3.4) and (3.5), for a suitable absolute constant $c_1$, we have
	\begin{align*}
		&\sum_{y^\delta<p_1,\cdots,p_{J+1}\leq y}\sum_{\substack{ \ell\leq x\\p_1\cdots p_{J+1}|\ell\\2|\ell }}\sum_{d|(\ell +1,P(z))}\lambda_d^+\\&\leq(2A_0+o(1))\frac{1}{\log D} \sum_{y^\delta<p_1,\cdots,p_{J+1}\leq y}\sum_{\substack{ \ell\leq x(2p_1\cdots p_{J+1})^{-1} }}H(\ell)\\&\leq  \left(c_1+o(1)\right)\frac{x}{\log x} \sum_{y^\delta<p_1,\cdots,p_{J+1}\leq y}\frac{1}{p_1\cdots p_{J+1}}\\&= (c_1+o(1))\frac{x}{\log x} \left(\log\frac{1}{\delta}\right)^{J+1}.
	\end{align*}
	
	To summarize all the above, we obtain
	\begin{align*}
		\mathscr{S}_1\leq \left(\frac{2}{5/8-\varepsilon}\rho\left(\frac{1}{c}\right)+c_1\delta+o(1)\right)\pi(x).\tag{3.7}
	\end{align*}

	\subsubsection{Estimate of $\mathscr{S}_2$} We will show that for sufficiently small $\delta$, $\mathscr{S}_2$ is small.
	Now we choose $D=x^{1/2}/(\log x)^{B}$.
	We have
	\begin{align*}
		\mathscr{S}_2&\leq \sum_{\substack{ p\leq x\\(p-1,P(y^\delta,y))=1}}1\leq \sum_{\substack{ p\leq x }}\sum_{d|(p-1,P(y^\delta,y))}\lambda_d^+= \sum_{d|P(y^\delta,y)}\lambda_d^+\sum_{\substack{ p\leq x\\p\equiv 1\mkern-15mu\pmod{d}}}1\\&=\pi(x)\sum_{d|P(y^\delta,y)}\frac{\lambda_d^+}{\varphi(d)}+\sum_{d|P(y^\delta,y)}\lambda_d^+\Bigg(\sum_{\substack{p\leq x\\p\equiv 1\mkern-15mu\pmod{d}}}1-\frac{\pi(x)}{\varphi(d)}\Bigg).
	\end{align*}
	By the Bombieri-Vinogradov theorem, we can bound the second summation by
	\begin{align*}
		\ll \frac{x}{(\log x)^A},
	\end{align*}
	which is negligible.
	For the main term, for a suitable absolute constant $c_2$, we have\begin{align*}
		\sum_{d|P(y^\delta,y)}\frac{\lambda_d^+}{\varphi(d)}\leq (F(2)+o(1))\prod_{\substack{y^\delta<p\leq y}}\left(\frac{p-2}{p-1}\right)\leq c_2\delta+o(1).
	\end{align*}
	Then we have
	\begin{align*}
		\mathscr{S}_2\leq (c_2\delta+o(1))\pi(x).\tag{3.8}
	\end{align*}
	
	Conclude from (3.1), (3.2), (3.7) and (3.8) that
	\begin{align*}
		\mathop{\lim\sup}_{x\rightarrow \infty}\frac{\mathcal{T}_c(x)}{\pi(x)}\leq \frac{2}{5/8-\varepsilon}\rho\left(\frac{1}{c}\right)+c_1\delta+c_2\delta.
	\end{align*}
	As $\max(\varepsilon,\delta)\rightarrow 0^+$, we complete the proof of this case.

	\subsection{Case 2: $11/32<c\leq1/2$}
	We have
	\begin{align*}
		\mathcal{T}_c(x)&=\sum_{\substack{p\leq x\\P^+(p-1)<x^c}}1=\sum_{\substack{p\leq x\\P^+(p-1)<x^c}}\frac{\log(p-1)}{\log x}+o(\pi(x))\\&=\sum_{\substack{p\leq x}}\frac{\log(h(p-1))}{\log x}-\sum_{\substack{p\leq x\\P^+(p-1)\geq x^c}}\frac{\log(h(p-1))}{\log x}+o(\pi(x))\\&:=\mathscr{T}_1(x)-\mathscr{T}_2(x)+o(\pi(x)),\tag{3.9}
	\end{align*}
	where $h(n):=\prod_{p^k\|n,p<x^c}p^k$.
	After invoking the identity
	\begin{align*}
		\log n=\sum_{m|n}\Lambda(m), 
	\end{align*}
	we have
	\begin{align*}
		\mathscr{T}_1(x)&= \frac{1}{\log x}\sum_{p'< x^c}\log p'\sum_{\substack{p\leq x\\p\equiv 1\mkern-15mu\pmod{p'}}}1+O\Bigg(\frac{1}{\log x}\sum_{k\geq 2}\sum_{p'^k\leq x}\log p'\sum_{\substack{p\leq x\\p\equiv 1\mkern-15mu\pmod{p'^k}}}1\Bigg)+o(\pi(x)).
	\end{align*}
	For the error term, we obtain the bound \(o(\pi(x))\). (To do this, split the condition \(p'^k \leq x\) into two ranges: \(p'^k \leq x^{2/3}\) and \(x^{2/3} < p'^k < x\), then use the Brun-Titchmarsh inequality for the first range, and trivial estimates for the second range.) For the main term, we use the Bombieri-Vinogradov theorem to get
	\begin{align*}
		\mathscr{T}_1(x)&=\frac{1}{\log x}\sum_{p'< x^c}\log p'\sum_{\substack{p\leq x\\p\equiv 1\mkern-15mu\pmod{p'}}}1+o(\pi(x))\\&= \frac{\pi(x)}{\log x}\sum_{p'< x^c}\frac{\log p'}{\varphi(p')}+o(\pi(x))\\&=\left(c+o(1)\right)\pi(x).\tag{3.10}
	\end{align*}
	
	For $\mathscr{T}_2(x)$, we have
	\begin{align*}
		\mathscr{T}_2(x)\geq &\frac{1}{\log x}\sum_{p\leq x}\Bigg(\sum_{\substack{x^{s}<p_1<x^{1/2}\\p_1|(p-1)}}\log(x/p_1)-\sum_{\substack{x^{s}<p_1<p_2<x^{1/2}\\p_1p_2|(p-1)}}\log x\\&-\sum_{\substack{x^{s}<p_1<x^{1/2}\\x^{c}<p_2<x^s\ \text{or}\ x^{1/2}<p_2<x/p_1\\p_1p_2|(p-1)}}\log p_2\Bigg)+o(\pi(x))\\= &\frac{1}{\log x}\sum_{x^s<p_1<x^{1/2}}\log(x/p_1)\sum_{\substack{p\leq x\\p\equiv 1\mkern-15mu\pmod{p_1}}}1\\&-\sum_{\substack{\ell,p_1,p_2\\\ell p_1p_2\leq x\\x^s<p_1<p_2<x^{1/2}\\\ell p_1p_2+1\ \text{is prime}}}1-\frac{1}{\log x}\sum_{\substack{\ell,p_1,p_2\\\ell p_1p_2\leq x\\x^s<p_1<x^{1/2}\\x^{c}<p_2<x^s\ \text{or}\ x^{1/2}<p_2<x/p_1\\\ell p_1p_2+1\ \text{is prime}}}\log p_2+o(\pi(x)),
	\end{align*}
	where $11/32<c<s<1/2$.
	Similar to the proof of [17, Theorem 1.1], we have
	\begin{align*}
		\mathscr{T}_2(x)\geq (g(s,c)+o(1))\pi(x),
	\end{align*}
	where
	\begin{align*}
		g(s,c):&=\log\frac{1}{2s}-\left(\frac{1}{2}-s\right)-\frac{8}{5}\left(\log\frac{1}{2s}\right)^2-\frac{16}{5}\left(s-c+\frac{1}{2}\right)\log\frac{1}{2s}+\frac{16}{5}\left(\frac{1}{2}-s\right)\\&=\frac{11}{5}\left(\frac{1}{2}-s\right)-\left(\frac{16}{5}\left(s-c+\frac{1}{2}\right)-1\right)\log\frac{1}{2s}-\frac{8}{5}\left(\log\frac{1}{2s}\right)^2.
	\end{align*}
	Then we have $g(1/2,c)=0$ and 
	\begin{align*}
		\frac{\partial}{\partial s}g(s,c)\Bigg|_{s=1/2}=\frac{11}{5}-\frac{32}{5}c<0.
	\end{align*}
	Thus for $11/32<c<1/2$, there exists $s_0\in(c,1/2)$ satisfying $g(s_0,c)>0$, i.e., there exists $j(c)>0$ such that
	\begin{align*}
		\mathscr{T}_2(x)\geq (j(c)+o(1))\pi(x).\tag{3.11}
	\end{align*}
	
	In particular, for $c=1/2$, we have
	\begin{align*}
		\mathscr{T}_2(x)=\frac{1}{\log x}\sum_{x^{1/2}\leq p'<x}\log(x/p')\sum_{\substack{p\leq x\\p\equiv 1\mkern-15mu\pmod{p'}}}1.
	\end{align*}
	By the following fact that (see [9, Theorem 1])
	\begin{align*}
		\frac{1}{\log x}\sum_{x^{1/2}\leq p'<x}\log p'\sum_{\substack{p\leq x\\p\equiv 1\mkern-15mu\pmod{p'}}}1=\left(\frac{1}{2}+o(1)\right)\pi(x),
	\end{align*}
	we have
	\begin{align*}
		\mathscr{T}_2(x)\geq& \frac{1}{\log x}\left(\frac{1}{s}-1\right)\sum_{x^{1/2}\leq p'<x}\log p'\sum_{\substack{p\leq x\\p\equiv 1\mkern-15mu\pmod{p'}}}1\\&-\frac{1}{\log x}\sum_{x^{s}\leq p'<x}\left(\frac{1}{s}\log p'-\log x\right)\sum_{\substack{p\leq x\\p\equiv 1\mkern-15mu\pmod{p'}}}1\\=&\frac{1}{2}\left(\frac{1}{s}-1\right)\pi(x)+o(\pi(x))-\frac{1}{\log x}\sum_{x^{s}\leq p'<x}\left(\frac{1}{s}\log p'-\log x\right)\sum_{\substack{p\leq x\\p\equiv 1\mkern-15mu\pmod{p'}}}1,
	\end{align*}
	where $1/2<s<1$.
	Similar to the proof of [17, Theorem 1.1], we also have
	\begin{align*}
		\frac{1}{\log x}\sum_{x^{s}\leq p'<x}\left(\frac{1}{s}\log p'-\log x\right)\sum_{\substack{p\leq x\\p\equiv 1\mkern-15mu\pmod{p'}}}1&=\frac{1}{\log x}\sum_{\substack{\ell,p'\\\ell p'\leq x\\x^s\leq p'<x\\\ell p'+1\ \text{is prime}}}\left(\frac{1}{s}\log p'-\log x\right)\\&\leq \Bigg(\frac{16}{5s}\left(1-s\right)-\frac{16}{5}\log\frac{1}{s}+o(1)\Bigg)\pi(x).
	\end{align*}
	Then by numerical calculation, we have
	\begin{align*}
		\mathscr{T}_2(x)\geq \Bigg(\frac{1}{2}\left(\frac{1}{s}-1\right)-\frac{16}{5s}\left(1-s\right)+\frac{16}{5}\log\frac{1}{s}+o(1)\Bigg)\pi(x)\geq (0.0436+o(1))\pi(x),\tag{3.12}
	\end{align*}
	by taking $s=0.843$.
	Concluding from (3.1), (3.9)--(3.12), we complete the proof of the second case.
	
	\section{Remark}
	For $1/2<\eta<17/32$, we have $\pi(x;q,1)\geq (C(\eta)+o(1))\pi(x)/\varphi(q)$ for almost primes $q\sim x^\eta$, where $C(\eta)>0$. (see [15, Lemma 2.2]) This result can help us to get a better lower bound for $\mathscr{T}_2(x)$. However, the computation of the function $C(\cdot)$ is somewhat cumbersome, so we omit it.

	We also have 
	\begin{align*}
		\mathscr{T}_2(x)\geq\frac{1}{\log x}\sum_{x^{1-c}<p'<x}\log(x/p')\sum_{\substack{p\leq x\\p\equiv 1\mkern-15mu\pmod{p'}}}1.
	\end{align*}
	By the following fact that
	\begin{align*}
		\frac{1}{\log x}\sum_{x^{1-c}<p'<x}\log p'\sum_{\substack{p\leq x\\p\equiv 1\mkern-15mu\pmod{p'}}}1\geq (f(c)+o(1))\pi(x),
	\end{align*}
	where $f(c)>0$ for $1/2\geq c\geq1-0.677=0.323$ (see [1, Theorem 2 and (7.1)]), we have
	\begin{align*}
		\mathscr{T}_2(x)\geq \left(g_1(s)+o(1)\right)\pi(x),
	\end{align*}
	where 
	\begin{align*}
		g_1(s):=f(0.323)\left(\frac{1}{s}-1\right)-\frac{16}{5s}\left(1-s\right)+\frac{16}{5}\log\frac{1}{s}.
	\end{align*}
	The function $g_1(s)$ satisfies $g_1(1)=0$ and $g_1'(1)=-f(0.323)<0$. Thus for $0.323\leq c\leq1/2$, there exists $s_0\in(1-c,1)$ satisfying $g_1(s_0)>0$. Unfortunately, the computation of the function $f(\cdot)$ is rather complicated; we only establish existence.
	
	\section*{Acknowledegements}

\end{document}